\documentclass[12pt]{article}
\usepackage{latexsym,amsfonts,amsmath,theorem,amssymb}
\usepackage{graphicx}

\usepackage{enumerate}
\usepackage{hyperref}

\newtheorem{theorem}{Theorem}

\newtheorem{corollary}[theorem]{Corollary}
\newtheorem{proposition}[theorem]{Proposition}

\newenvironment{proof}{\begin{trivlist}
    \item[\hskip\labelsep{\bf Proof.}]}{$\hfill\Box$\end{trivlist}}

{\theoremstyle{plain} \theorembodyfont{\rmfamily}
\newtheorem{remark}[theorem]{Remark}}
{\theoremstyle{plain} \theorembodyfont{\rmfamily}
\newtheorem{example}[theorem]{Example}}

\allowdisplaybreaks

\newcommand{\bsgamma}{{\boldsymbol{\gamma}}}

\newcommand{\bsh}{{\boldsymbol{h}}}

\newcommand{\bst}{{\boldsymbol{t}}}

\newcommand{\bsx}{{\boldsymbol{x}}}

\newcommand{\bsy}{{\boldsymbol{y}}}
\newcommand{\bsz}{{\boldsymbol{z}}}

\newcommand{\bszero}{{\boldsymbol{0}}}

\newcommand{\rd}{\,\mathrm{d}}

\newcommand{\bbR}{\mathbb{R}}

\newcommand{\bbN}{\mathbb{N}}

\newcommand{\calI}{\mathcal{I}}

\newcommand{\mask}[1]{}
\newcommand{\esup}{\operatornamewithlimits{ess\,sup}}

\newcommand{\e}{{\varepsilon}}
\newcommand{\setu}{{\mathfrak{u}}}
\newcommand{\setv}{{\mathfrak{v}}}

\newcommand{\setU}{{\mathfrak{U}}}

\newcommand{\norm}[1]{\left\Vert#1\right\Vert}
\newcommand{\abs}[1]{\left\vert#1\right\vert}
\newcommand{\sset}{[s]}

\title{
Very Low Truncation Dimension for High Dimensional Integration Under
Modest Error Demand}
\author{Peter Kritzer\thanks{P. Kritzer is supported by the Austrian
Science Fund (FWF):
Project F5506-N26, which is a part of the Special Research Program
"Quasi-Monte Carlo Methods:
Theory and Applications".} ,
Friedrich Pillichshammer\thanks{F. Pillichshammer is
partially supported by the Austrian Science Fund (FWF): Project F5509-N26,
which is a part of the Special Research Program "Quasi-Monte Carlo Methods:
Theory and Applications".}\;, and G. W. Wasilkowski}

\date{}

\begin{document}
\maketitle

\centerline{\it Dedicated to the memory of Joseph F. Traub (1932-2015)}

\begin{abstract}
We consider the problem of numerical integration for
weighted anchored and ANOVA Sobolev spaces of $s$-variate functions.
Here $s$ is large including $s=\infty$. Under the assumption of
sufficiently fast decaying weights, we prove in a constructive way that
such integrals can be approximated by quadratures for functions $f_k$
with only $k$ variables, where $k=k(\e)$ depends solely
on the error demand $\e$ and is surprisingly small when $s$ is
sufficiently large relative to $\e$.
This holds, in particular, for $s=\infty$ and arbitrary
$\e$ since then $k(\e)<\infty$ for all $\e$.
Moreover $k(\e)$ does not depend on the function being integrated, i.e.,
is the same for all functions from the unit ball of the space.
\end{abstract}

\centerline{\begin{minipage}[hc]{130mm}{
{\em Keywords:} numerical integration, weighted anchored and ANOVA
Sobolev spaces, truncation dimension  \\
{\em MSC 2000:}  65D30, 65D32}
\end{minipage}}

\section{Introduction}
This paper has been inspired by \cite{DKLP15}. There the classical
multivariate integration problem of approximating
\[
\calI_s(f)\,=\,\int_{[0,1]^s} f(\bsx)\rd\bsx\quad\mbox{for large\ }s
\]
was considered for functions from $\bsgamma$-weighted Sobolev spaces
of functions with mixed derivatives of order one bounded in $L_2$ norm.
Such spaces are Hilbert spaces and have been assumed in a number of papers
dealing with multivariate integration. In particular, there is a number of
papers, see, e.g., \cite{kor,kuo,NC06,CoolNuy,SKJ,SR}, initiated by
the work in \cite{kor}, that study
{\em Component-By-Component} ({\em CBC}, for short) methods of constructing
efficient lattice rules for approximating $\calI_s(f)$. The authors of
\cite{DKLP15} proved that for rapidly decreasing weights, one can
decrease the cost of fast CBC by restricting the search space
for variables with smallest weights;
however still dealing with $s$-variate integrals.

This observation motivated us to consider efficient dimension truncation
and compare it to the technique from \cite{DKLP15}. Dimension
truncation has been considered in a number of papers; however,
mainly for specific integrands. Our approach is in the {\em worst case setting}
spirit, i.e., we study truncation that depends only on the error
demand $\e$ and global properties of the Banach space of integrands.
We consider more general weighted integration problems defined on more general
classes of functions, and the proposed truncation technique does not depend
on specific algorithms. Moreover, for some classes of weights,
when the corresponding integration problem is well defined, see, e.g.,
\cite{GnHeHiRiWa15,HeRiWa15,SH,Was14}, $s$ could be infinite.

More precisely, we consider approximating
\[
  \calI_s(f)\,=\,\int_{D^s} f(\bsx)\,\rho_s(\bsx)\rd\bsx,
\]
where $D$ can be an arbitrary (bounded or unbounded) interval, $\rho$ is
a probability density function on $D$, and
\[
\rho_s(\bsx)\,:=\,\prod_{j=1}^s\rho(x_j)\ \ \ \mbox{ for $\bsx=(x_1,\ldots,x_s)$}.
\]
As for the spaces of integrands $f$, we consider both {\em anchored}
and {\em ANOVA} spaces (denoted respectively by $F_{s,p,\bsgamma}$
and $H_{s,p,\bsgamma}$) of functions with mixed first order derivatives
bounded in $\psi$-weighted $L_p$ norm. These are Banach (as opposed to
Hilbert) spaces and were considered in a number of papers including
\cite{GnHeHiRiWa15,Was14}. Their definitions and properties are recalled
in Section 2. Here we briefly show the norm for the anchored case
\[
\|f\|_{F_{s,p,\bsgamma}}\,=\,
\left(\sum_\setu \gamma_\setu^{-p}\,\int_{D^{|\setu|}}
  |f^{(\setu)}([\bsx_\setu;\bszero_{-\setu}])|^p\prod_{j\in\setu}
  (\psi(x_j)\rd x_j)\right)^{1/p}.
\]
Here $\psi$ is a positive (a.e., on $D$) probability density function,
the summation is with respect to the subsets $\setu$ of
$[s]=\{1,\dots,s\}$ (when $s=\infty$ the summation
is with respect to all finite subsets of $\bbN$),
and $f^{(\setu)}([\bsx_\setu;\bszero_{-\setu}])$ denotes the
mixed partial derivatives $\prod_{j\in\setu}\frac{\partial}{\partial x_j}$
of $f$ with values of $x_j$ for $j\notin\setu$ being zero. Finally,
$\gamma_\setu$ are non-negative real
numbers that quantify the importance of sets
$\bsx_\setu=(x_j)_{j\in\setu}$ of variables.

We show that for the anchored spaces with large $s$,
if the weights decay sufficiently fast then it is possible to approximate
the original $s$-variate integral $\calI_s(f)$ by the following
$k$-variate integral
\[
\calI_k(f_k),\quad\mbox{where}\quad  f_k(x_1,\dots,x_k)\,:=\,
f(x_1,\dots,x_k,0,\dots,0).
\]
This is because for a modest error demand $\e$, one can compute $k=k(\e)$
such that
\begin{equation}\label{important}
k(\e)\,\ll\,s\quad\mbox{and}\quad
|\calI_s(f)-\calI_{k(\e)}(f_{k(\e)})|\,\le\,\frac{\e}{2^{1-1/p}}\,
\|f\|_{F_{s,p,\bsgamma}}\quad\mbox{for all\ }f\,\in\,F_{s,p,\bsgamma}.
\end{equation}

Actually, the truncation error in \eqref{important} is slightly smaller
\[
|\calI_s(f)-\calI_k(f_k)|\,\le\,\frac{\e}{2^{1-1/p}}\,
\|f-f_k\|_{F_{s,p,\bsgamma}}.
\]
Therefore for any $k$-variate rule $A_k$ for approximating integrals from
the space $F_{k,p,\bsgamma}$ with the worst case error bounded by
\[
  e(A_k;F_{k,p,\bsgamma})\,\le\,\frac{\e}{2^{1-1/p}},
\]
the resulting $s$-variate quadrature
$Q_{s,p,\bsgamma}^{\rm trnc}$, defined by
\[
Q_{s,p,\bsgamma}^{\rm trnc}(f)\,:=\,A_k(f_k),
\]
has its worst case error bounded by
\[
e(Q_{s,p,\bsgamma}^{\rm trnc};F_{s,p,\bsgamma})\,\le\,\e.
\]

Since for modest values of $\e$ the truncation dimension $k(\e)$
is small, the approach suggested in the current paper
could lead to very efficient ways of dealing with integrals that have
a huge (including $\infty$) number of variables.
We illustrate this for the classical integration problem
($D=[0,1]$ and $\psi=\rho\equiv1$) and special product weights
\[
\gamma_\setu\,=\,\prod_{j\in\setu}j^{-a}\quad\mbox{with}\quad a>1.
\]
(For such weights the space $F_{s,p,\bsgamma}$ and the integration problem
are well defined even for $s=\infty$, see, e.g., \cite{Was14}.)
Then
\[
 k(\e)\,=\,O\left(\e^{-1/(a-1+1/p)}\right).
\]
For instance for $p=2$ and $a=2$, some values of $k(\e)$ are listed below
\[
\begin{array}{c||c|c|c|c}
  \e& 10^{-1}& 10^{-2}& 10^{-3} &10^{-4}\\
  \hline
  k(\e)&4 & 17      & 77     &357
  \end{array},
\]
and they are surprisingly small for modest error demand $\e$.
They are even smaller for $p=1$:
\[
\begin{array}{c||c|c|c|c}
  \e& 10^{-1}& 10^{-2}& 10^{-3} &10^{-4}\\
  \hline
  k(\e)&3    & 9      & 31      & 99
  \end{array}.
\]
More values of $k(\e)$ are provided in Example \ref{ex:1} and it is
clear that the larger $a$ or the smaller $p$ the smaller values of
$k(\e)$.

We also show there that for the {\em normalized worst case error},
the corresponding values, denoted by $k^{\rm nrm}(\e)$, are even
smaller. For instance for $p=2=a$, we have
\[
\begin{array}{c||c|c|c|c}
\e &  10^{-1}  &  10^{-2}  & 10^{-3}  & 10^{-4}\\
\hline
k^{\rm nrm}(\e)& 3      &  15      &  69      & 322
\end{array}.
\]

These results depend very much on special properties of anchored spaces and
do not hold in general for ANOVA spaces with arbitrary weights. However,
for product weights
\[
\gamma_\setu\,=\,\prod_{j\in\setu}\gamma_j
\]
we can use the following fact due to \cite{GnHeHiRiWa15}.
The corresponding anchored and ANOVA spaces are equal
(as sets of functions)  and their norms
are equivalent with the equivalence constant bounded by
\[
   \prod_{j=1}^s(1+\gamma_j\,\kappa)
\]
for a number $\kappa$ that depends on $p$ and $\psi$.
(Note that if $\sum_{j=1}^\infty\gamma_j<\infty$ then the equivalence
hold also for $s=\infty$.)
Of course, when this constant is not too large, efficient algorithms for
the integrands from $F_{s,p,\bsgamma}$ are also efficient for
integrands from $H_{s,p,\bsgamma}$.

To simplify the presentation, from Sections 2 to 6, we deal with the
classical integration problem and finite $s$.
Basic concepts are presented in Section 2. The results on the
dimension truncation for anchored spaces are in Section 3.
In Section 4, we apply recent results, see \cite{HeRi13,HeRiWa15,SH},
on the equivalence of anchored and ANOVA spaces for product weights so that the
truncation technique from Section 3 can be used for ANOVA spaces.
Moreover, we improve the result of \cite{HeRi13} by
providing in Theorem \ref{thm:p2} the exact value of the embedding
operator for $p=2$ from the anchored onto ANOVA space.
We next use the above truncation results in
Sections 5 and 6 to derive a more efficient fast CBC algorithm.
Generalizations of results from Sections 3 and 4 are presented briefly
in Section 7; they rely on the same proof techniques.

Finally we want to add that the worst case approach to the effective
dimension in both the truncated and superposition sense is considered
in \cite{owen14}; however, only for weighted Hilbert spaces ($p=2$)
of periodic integrands with $D=[0,1]^s$.  Moreover, the effective
dimension is defined in terms of
variances of the components from (classical) ANOVA decomposition
of functions.
Although the results of \cite{owen14} are very interesting from
theoretical point of view, they are
not explicitly related to the errors of algorithms.

\section{Basic Concepts}

\subsection{Anchored and Unanchored Spaces}\label{defspace}

In this section, we introduce the basic definitions of the anchored and
unanchored (ANOVA) Sobolev spaces of $s$-variate functions. More
detailed information can be found in \cite{HeRiWa15,SH,Was14}.

Here we follow \cite[Section~2]{Was14}: For $p \in [1,\infty]$
let $F=W_{p,0}^1([0,1])$ be the space of functions defined on $[0,1]$
that vanish at zero, are absolutely continuous, and have bounded
derivative in the $L_p$ norm. We endow $F$ with the norm
$\|f\|_F=\|f'\|_{L_p}$ for $f \in F$.

For $s \in \mathbb{N}$ and
\[
  \sset\,:=\,\{1,2,\dots,s\},
\]
we will use $\setv,\setu$ to denote
subsets of $\sset$, i.e.,
\[
  \setv,\setu\subseteq\,\sset.
\]
Moreover, for $\bsx=(x_1,x_2,\ldots,x_s)\in[0,1]^s$ and
$\setu\subseteq\sset$, $[\bsx_\setu;\bszero_{-\setu}]$ denotes the
$s$-dimensional vector with all $x_j$ for $j\notin\setu$ replaced
by zero, i.e.,
\[
  [\bsx_\setu;\bszero_{-\setu}]\,=\,(y_1,y_2,\dots,y_s)\quad
   \mbox{with}\quad  y_j\,=\,\left\{\begin{array}{ll} x_j &
    \mbox{if\ }j\in\setu,\\ 0 & \mbox{if\ } j\notin\setu.\end{array}
    \right.
\]
We also write $\bsx_\setu$ to denote the $|\setu|$-dimensional
vector $(x_j)_{j\in\setu}$ and
\[
  f^{(\setu)}\,=\,\frac{\partial^{|\setu|} f}{\partial \bsx_{\setu}}\,
=\,\prod_{j\in\setu}\frac{\partial}{\partial x_j}\,f
  \quad\mbox{with}\quad f^{(\emptyset)}=f.
\]

For $s \in \mathbb{N}$ and nonempty $\setu \subseteq [s]$ let $F_{\setu}$
be the completion of the space spanned by
$f(\bsx)=\prod_{j \in \setu} f_j(x_j)$ for $f_j \in F$
and $\bsx=(x_1,\ldots,x_s) \in [0,1]^s$, with the norm
$$\|f\|_{F_{\setu}}=\|f^{(\setu)}\|_{L_p}.$$
Note that $F_{\setu}$ is a
space of functions with domain $[0,1]^s$ that depend only on the
variables listed in
$\setu$. For $\setu=\emptyset$, let $F_{\setu}$ be the
space of constant functions with the natural norm.

Consider next a sequence $\bsgamma=(\gamma_\setu)_{\setu\subseteq[s]}$
of non-negative real numbers, called {\em weights}. Since some
weights could be zero, we will use
\[
  \setU\,=\,\{\setu\subseteq\sset\,:\,\gamma_\setu>0\}
\]
to denote the collection of positive weights. For $p\in[1,\infty]$,
we define
\[
F_{s,p}= {\rm span} \left(\bigcup_{\setu \in \setU} F_{\setu}\right).
\]
The corresponding weighted {\em anchored} space $F_{s,p,\bsgamma}$
is the completion of $F_{s,p}$ with respect to the norm
\[
  \|f\|_{F_{s,p,\bsgamma}}\,=\,\left\{
\begin{array}{ll}
 \left(\sum_{\setu\subseteq\sset}\frac{1}{\gamma_\setu^p}\,
    \|f^{(\setu)}([\cdot_\setu;\bszero_{-\setu}])\|_{L_p}^p
   \right)^{1/p} & \mbox{ if $p < \infty$,}\\[0.8em]
 \max_{\setu\in [s]}\frac1{\gamma_\setu}\,
   \esup_{\bsx_\setu\in[0,1]^{|\setu|}}
   \left|f^{(\setu)}([\bsx_\setu;\bszero_{-\setu}])\right| &
   \mbox{ if $p = \infty$.}
\end{array}\right.
\]
For $\gamma_\setu=0$, the corresponding term
$f^{(\setu)}([\cdot_\setu;\bszero_{-\setu}])\equiv0$.
We then have
\[
 \|f\|_{F_{s,p,\bsgamma}}\,=\,\left(\sum_{\setu\in\setU}\frac{1}{\gamma_\setu^p}\,
    \|f^{(\setu)}([\cdot_\setu;\bszero_{-\setu}])\|_{L_p}^p
   \right)^{1/p}.
\]
For $p=\infty$ the norm reduces to
\[
   \|f\|_{F_{s,\infty,\bsgamma}}\,=\,\max_{\setu\in\setU}\frac1{\gamma_\setu}\,
   \esup_{\bsx_\setu\in[0,1]^{|\setu|}}
   \left|f^{(\setu)}([\bsx_\setu;\bszero_{-\setu}])\right|.
\]
More information on the structure of the space $F_{s,p,\bsgamma}$ can
be found in \cite[Section~2]{Was14}.

An important class of weights is provided by product weights
\[
  \gamma_\setu\,=\,\prod_{j\in\setu}\gamma_j
\]
for positive reals $\gamma_j$. When dealing with them, we will assume
without any loss of generality that
\[
  1\ge \gamma_1,\quad\mbox{and}\quad
  \gamma_j\,\ge\,\gamma_{j+1}\,>\,0\quad\mbox{for all\ }j.
\]
Note that for product weights we have
$\setU=2^{\sset}=\left\{\setu: \setu\subseteq\sset\right\}$.

For $p=2$, $F_{s,2,\bsgamma}$ is a reproducing kernel Hilbert
space with kernel
\[
  K'(\bsx,\bsy)\,=\,\sum_{\setu\in\setU}\gamma_\setu^{2}\prod_{j\in\setu}
  \min(x_j,y_j),
\]
for $\bsx=(x_1,\ldots,x_s)$ and analogously for $\bsy$, which for
product weights reduces to
\[
  K'(\bsx,\bsy)\,=\,\prod_{j=1}^s\left(1+\gamma_j^2\,\min(x_j,y_j)\right).
\]

The weighted {\em unanchored} (or {\em ANOVA}) Sobolev space
$H_{s,p,\bsgamma}$ is the Banach space of continuous functions
$f:[0,1]^s \to\bbR$ with finite norm
\[
  \|f\|_{H_{s,p,\bsgamma}}=\left(\sum_{\setu\in\setU}\frac{1}{\gamma_{\setu}^p}
   \,\left\|\int_{[0,1]^{s-|\setu|}}
     f^{(\setu)}([\cdot_\setu;\bsx_{-\setu}])
     \rd \bsx_{-\setu}\right\|^p_{L_p}\right)^{1/p}.
\]
For $p=\infty$ the norm reduces to
\[
  \|f\|_{H_{s,\infty,\bsgamma}}=\max_{\setu\in\setU}\frac1{\gamma_\setu}\,
   \esup_{\bsx_\setu\in
   [0,1]^{|\setu|}}\left|\int_{[0,1]^{s-|\setu|}}
     f^{(\setu)}([\bsx_\setu;\bsx_{-\setu}])
     \rd \bsx_{-\setu}\right|.
\]
For $p=2$ and product weights,
$H_{s,2,\bsgamma}$ is a reproducing kernel Hilbert space with
kernel function given by
\[
  K(\bsx,\bsy)\,=\, \prod_{j=1}^s K_{\gamma_j}(x_j, y_j)
  \,=\, \prod_{j=1}^s (1 + \gamma_j^2(\tfrac{1}{2}B_2(\{x_j-y_j\})
  + (x_j-\tfrac{1}{2})(y_j-\tfrac{1}{2}))),
\]
where $B_2(x) = x^2 - x + \tfrac{1}{6}$ is the second Bernoulli
polynomial and $\{x\} = x - \lfloor x \rfloor$.

\subsection{Algorithms and  Errors}\label{secalgdef}
We consider algorithms that use a finite number $n$ of samples
$f(\bsx_i)$. Without loss of generality, see, e.g., \cite{TWW88},
we can restrict the attention to linear algorithms, called
quadratures,
\[
  Q_{s,n}(f)\,=\,\sum_{i=1}^n a_i\,f(\bsx_i)
\]
for $a_i\in\bbR$ and $\bsx_i\in[0,1]^s$. An important class of
quadratures is provided by {\em quasi-Monte Carlo} methods with
all coefficients $a_i=1/n$ (see, e.g., \cite{DKS, DP10, LP14, niesiam}).

We consider in this paper the {\em worst case} error defined by
\[
   e(Q_{s,n};F_{s,p,\bsgamma})\,=\,\|\calI_s-Q_{s,n}\|\,=\,
  \sup_{\|f\|_{F_{s,p,\bsgamma}}\le1}|\calI_s(f)-Q_{s,n}(f)|.
\]
It is well known that the operator norm
$\|\calI_s\|=\sup_{\|f\|_{F_{s,p,\bsgamma}}\le 1}|\calI_s(f)|$ of
$\calI_s$ is equal to
\begin{equation}\label{I-nrm}
   \|\calI_s\|
  \,=\,\left\{\begin{array}{ll}
  \left(\sum_{\setu\in\setU}\frac{\gamma_\setu^{p^*}}
   {(p^*+1)^{|\setu|}}\right)^{1/p^*} & \mbox{for\ }p>1,\\
  \max_{\setu\in\setU}\gamma_\setu & \mbox{for\ }p=1, \end{array}\right.
\end{equation}
see, e.g., \cite{Was14}. Here $p^*$ is the conjugate of $p$,
i.e.,
\[
  \frac1p+\frac1{p^*}\,=\,1.
\]

In the case of product weights the formula \eqref{I-nrm} can be
rewritten to
\begin{equation*}\label{I-nrm-prod}
   \|\calI_s\|
  \,=\,\left\{\begin{array}{ll}
  \prod_{j=1}^s\left(1+\frac{\gamma_j^{p^*}}{p^* +1} \right)^{1/p^*}
& \mbox{for\ }p>1.\\
  \max_{\setu\subseteq[s]} \prod_{j\in\setu}\gamma_j & \mbox{for\
  }p=1,
\end{array}\right.
\end{equation*}
Of course, since we assumed that all $\gamma_j\le1$ then the maximum
above is attained by $\setu=\emptyset$ and is in fact equal to 1.

The definitions of the errors for the space $H_{s,p,\bsgamma}$ are
similar. The norm of the integration operator with respect to the
space $H_{s,p,\bsgamma}$ is also equal to the right hand side
of~\eqref{I-nrm}.

\section{Anchored Decomposition and Truncation}
It is well known, see, e.g., \cite{KSWW10a}, that any $f\in
F_{s,p,\bsgamma}$ has the unique {\em anchored decomposition}
\begin{equation}\label{dec-a}
  f\,=\,\sum_{\setu\in\setU} f_\setu,
\end{equation}
where $f_\setu$ is an element of $F_\setu$,
depends only on $x_j$ for $j\in\setu$, and
\begin{equation}\label{anc}
   f_\setu(\bsx)\,=\,0\quad\mbox{if\ }x_j=0\mbox{\ for some\ }j\in\setu.
\end{equation}
For the empty set $\setu$, $f_\emptyset$ is a constant function.
We stress that in general we do not know what the elements
$f_\setu$ are and we can only evaluate the original function $f$.

The anchored  decomposition has the following important
properties, see, e.g., \cite{HeRiWa15}:
\begin{equation}\label{der}
   f^{(\setu)}([\cdot_\setu;\bszero_{-\setu}])\,\equiv\,f_{\setu}^{(\setu)}.
\end{equation}
Due to \eqref{anc} we have
\[
f_\setu\,\equiv\,0\quad\mbox{iff}\quad
   f^{(\setu)}([\cdot_\setu;\bszero_{-\setu}])\,\equiv\,0,
\]
and due to \eqref{dec-a} and \eqref{der}
\[
   \|f\|_{F_{s,p,\bsgamma}}\, = \, \left\| \sum_{\setu\in\setU}
     f_\setu \right\|_{F_{s,p,\bsgamma}} \,=\,
  \left(\sum_{\setu\in\setU}\gamma_\setu^{-p}\,
    \|f_\setu^{(\setu)}\|_{L_p}^p\right)^{1/p}\quad\mbox{for\ }
    p<\infty
\]
and
\[
  \|f\|_{F_{s,\infty,\bsgamma}}\,=\,\max_{\setu\in\setU}
   \frac{\|f_\setu^{(\setu)}\|_{L_\infty}}{\gamma_\setu}
   \quad\mbox{for\ }p=\infty.
\]
For any $\setu\not=\emptyset$, there exists (unique in $L_p$-sense)
$g \in L_p([0,1]^{|\setu|})$ such that
\[
  f_\setu(\bsx)\,=\,\int_{[0,1]^{|\setu|}}g(\bst)\,\prod_{j\in\setu}1_{[0,x_j)}(t_j)
  \rd \bst\quad\mbox{and}\quad
  f_\setu^{(\setu)}\,=\,g,
\]
where $1_J(t)$ is the characteristic function of the set $J$,
i.e., $1_J(t)=1$ if $t \in J$ and 0 otherwise.

Moreover, for any $\setu$,
\[
  f([\cdot_\setu;\bszero_{-\setu}])\,=\,\sum_{\setv\subseteq\setu}f_\setv.
\]
In particular, for $k<s$ we have
\begin{equation}\label{trnk}
  f([\bsx_{[k]};\bszero_{-[k]}])\,=\,f(x_1,\dots,x_k,0,\dots,0)
  \,=\,\sum_{\setv\subseteq[k]}f_\setv(\bsx)
\end{equation}
which allows us to compute samples  and approximate the integral
of the {\em truncated} function
\[
  f_k(x_1,\dots,x_k)\,=\,\sum_{\setu\subseteq[k]}f_\setu(\bsx).
\]
Moreover, $f_k\in F_{k,p,\bsgamma}\subset F_{s,p,\bsgamma}$ and
\[
\|f_k\|_{F_{k,p,\bsgamma}}\,=\,\|f([\cdot_{[k]};\bszero_{-[k]}])\|_{F_{s,p,\bsgamma}}\,=\,
\left\|\sum_{\setu\subseteq[k]}f_\setu
   \right\|_{F_{s,p,\bsgamma}}.
\]

For given $k\in\sset$, let $A_{k,n}$ ($n\in\bbN$) be a family of
algorithms to approximate integrals
\[
  \calI_k(g)\,=\,\int_{[0,1]^k}g(\bsx)\rd \bsx
\]
for functions from the space $F_{k,p,\bsgamma}$. We use them to
define the following quadratures for the original space
$F_{s,p,\bsgamma}$
\begin{equation}\label{trnk-q}
  Q^{\rm trnc}_{s,n,k}(f)\,=\,A_{k,n}(f([\cdot_{[k]};\bszero_{-[k]}]))
  \,=\,A_{k,n}(f_k).
\end{equation}
Clearly, the quadratures $Q^{\rm trnc}_{s,n,k}$ are well defined.

We have the following result.

\begin{theorem}\label{thm-main}
For every $k \in \sset$ the worst case error of $Q_{s,n,k}^{\rm
trnc}$ is bounded by
\[
  e(Q^{\rm trnc}_{s,n,k};F_{s,p,\bsgamma})\,\le\,
  \left( [e(A_{k,n};F_{k,p,\bsgamma})]^{p^*}  + \sum_{\setu\not\subseteq[k]}
\frac{\gamma_\setu^{p^*}\,}{(p^*+1)^{|\setu|}} \right)^{1/p^*}
\quad\mbox{for\ }p>1
\]
and by
\[
 e(Q^{\rm trnc}_{s,n,k};F_{s,1,\bsgamma})\,\le\,
   \max\left(e(A_{k,n};F_{k,1,\bsgamma})\,,\,
     \max_{\setu\not\subseteq[k]}\gamma_\setu\right) \quad\mbox{for\ }p=1,
\]
where in the case $k=s$ we set
$\max_{\setu\not\subseteq\sset}\gamma_\setu:=0$.
\end{theorem}

\begin{proof}
We prove the theorem for $p>1$ only since the proof for $p=1$ is
very similar. For any $f\in F_{s,p,\bsgamma}$
\begin{eqnarray*}
   \left|\calI_s(f)-Q^{\rm trnc}_{s,n,k}(f)\right|&=&
   \left|\calI_k(f([\cdot_{[k]};\bszero_{-[k]}]))
    -A_{k,n}(f([\cdot_{[k]};\bszero_{-[k]}])) +
   \sum_{\setv\not\subseteq[k]}\calI_s(f_\setv)\right|\\
  &\le& e(A_{k,n};F_{k,p,\bsgamma})\,\|f_k\|_{F_{k,p,\bsgamma}}+
   \sum_{\setv\not\subseteq[k]}|\calI_s(f_\setv)|.
\end{eqnarray*}
Since $\calI_s$ has the tensor product form and $f_\setv$ depends
only on $|\setv|$ variables,
\[
  |\calI_s(f_\setv)|\,\le\,\|f_\setv\|_{F_\setv}
  \left(\sup_{\substack{f\in F_{\{1\}}\\ \norm{f}_{F_{\{1\}}}\le 1}}
  \abs{\calI_1 (f)}\right)^{\abs{\setv}}
  \,=\,\|f_{\setv}^{(\setv)}\|_{L_p}\,\frac1{(p^*+1)^{|\setv|/p^*}},
\]
where we used that $\sup_{\substack{f\in F_{\{1\}}\\ \norm{f}_{F_{\{1\}}}\le 1}}
  \abs{\calI_1 (f)}=\frac1{(p^*+1)^{1/p^*}}$, which can be checked easily.
Therefore
\begin{eqnarray*}
 \sum_{\setv\not\subseteq[k]}|\calI_s(f_\setv)|
    &\le&\sum_{\setv\not\subseteq[k]}\gamma_\setv^{-1}\,
   \|f_{\setv}^{(\setv)}\|_{L_p}
   \,\frac{\gamma_{\setv}}{(p^*+1)^{|\setv|/p^*}}\\
 &\le& \left(\sum_{\setv\not\subseteq[k]}\gamma_\setv^{-p}\,
  \|f_\setv^{(\setv)}\|_{L_p}^p\right)^{1/p}
   \,\left(\sum_{\setv\not\subseteq[k]}\frac{\gamma_\setv^{p^*}}{(p^*+1)^{|\setv|}}
   \right)^{1/p^*}.
\end{eqnarray*}
Hence putting together, we get
\begin{eqnarray*}
  |\calI_s(f)-Q^{\rm trnc}_{s,n,k}(f)|&\le&
  e(A_{k,n};F_{k,p,\bsgamma})\,\left(\sum_{\setu\subseteq[k]}\gamma_\setu^{-p}\,
   \|f_\setu^{(\setu)}\|_{L_p}^p\right)^{1/p}\\
  &&\quad +
   \left(\sum_{\setu\not\subseteq[k]}\frac{\gamma_\setu^{p^*}}{(p^*+1)^{|\setu|}}
   \right)^{1/p^*}\,\left(\sum_{\setu\not\subseteq[k]}\gamma_\setu^{-p}\,
    \|f_\setu^{(\setu)}\|_{L_p}^p\right)^{1/p}\\
\end{eqnarray*}
Finally, using the H\"older inequality one more time, we get
\begin{eqnarray*}
\lefteqn{|\calI_s(f)-Q^{\rm trnc}_{s,n,k}(f)|}\\
  &\le& \left(\sum_{\setu\subseteq[s]}\gamma_\setu^{-p}\,
   \|f_\setu^{(\setu)}\|_{L_p}^p\right)^{1/p}\,
   \left([e(A_{k,n};F_{k,p,\bsgamma})]^{p^*}+
   \sum_{\setu\not\subseteq[s]}\frac{\gamma_\setu^{p^*}}{(p^*+1)^{|\setu|}}
   \right)^{1/p^*}.
\end{eqnarray*}
This completes the proof.
\end{proof}

We now apply this theorem to product weights.
First we prove an upper bound on the truncation error.

\begin{proposition}\label{proptrunc}
Consider product weights $\gamma_\setu=\prod_{j\in\setu}\gamma_j$ and
$k\le s$. The truncation error is bounded by
\[
\left(\sum_{\setu\not\subseteq[k]}\frac{\gamma_\setu^{p^*}}{(p^*+1)^{|\setu|}}
\right)^{1/p^*}
 \,\le\, \|\calI_s\|
\left(1-\exp\left(\frac{-3}{2\,(p^*+1)}\sum_{j=k+1}^s \gamma_j^{p^*}
\right)\right)^{1/p^*}\quad\mbox{for\ }p>1,
\]
and it is equal to
\[
  \max_{\setu\not\subseteq[k]}\gamma_\setu \quad\mbox{for\ }p=1.
\]
\end{proposition}

\begin{proof}
The proof for $p=1$ is trivial. For $p>1$, we have
\begin{eqnarray*}
  \sum_{\setu\not\subseteq[k]}\frac{\gamma_\setu^{p^*}}{(p^*+1)^{|\setu|}}
   &=&\sum_{\setu\subseteq\sset}\frac{\gamma_\setu^{p^*}}{(p^*+1)^{|\setu|}}
   -\sum_{\setu\subseteq[k]}\frac{\gamma_\setu^{p^*}}{(p^*+1)^{|\setu|}}\\
   &=&\prod_{j=1}^s\left(1+\frac{\gamma_j^{p^*}}{p^*+1}\right)
   -\prod_{j=1}^k\left(1+\frac{\gamma_j^{p^*}}{p^*+1}\right)\\
   &=& \prod_{j=1}^s\left(1+\frac{\gamma_j^{p^*}}{p^*+1}\right)
     \left(1-\prod_{j=k+1}^s \frac{p^* +1}{p^* +1
    + \gamma_j^{p^*}}\right)\\    &=&\|\calI_s\|^{p^*}
  \left(1-\prod_{j=k+1}^s \frac{p^* +1}{p^* +1 + \gamma_j^{p^*}}\right).
\end{eqnarray*}
We have
\begin{eqnarray*}
 1-\prod_{j=k+1}^s \frac{p^* +1}{p^* +1 + \gamma_j^{p^*}}
&=& 1-\exp \left(\sum_{j=k+1}^s \log\frac{p^* +1}{p^*+1+\gamma_j^{p^*}}\right)\\
&=&1-\exp\left(\sum_{j=k+1}^s
\log\left(1-\frac{\gamma_j^{p^*}}{p^*+1+\gamma_j^{p^*}}\right)\right),
\end{eqnarray*}
where $\log$ denotes the natural logarithm.
Note that for $x\in[0,1/2]$,
\begin{eqnarray*}
\log(1-x)&=&-x\,\left(1+\frac{x}2+\frac{x^2}3+\cdots\right)\,\ge\,
       -x\,\left(1+\frac{x}2\left(1+x+x^2+\cdots\right)\right)\\
         &=&-x\,\left(1+\frac{x}{2\,(1-x)}\right)\,\ge\,-\frac{3\,x}2.
\end{eqnarray*}
Since we assumed that the product weights are bounded by 1, we can
apply this estimate to the above expression, and therefore
\begin{eqnarray*}
1-\prod_{j=k+1}^s \frac{p^* +1}{p^* +1 + \gamma_j^{p^*}}&\le&
  1-\exp\left(\frac{-3}{2}\,\sum_{j=k+1}^s
\frac{\gamma_j^{p^*}}{p^*+1+\gamma_j^{p^*}}\right)\\
&\le&  1-\exp\left(\frac{-3}{2\,(p^*+1)}\,\sum_{j=k+1}^s
   \gamma_j^{p^*}\right).
\end{eqnarray*}
This completes the proof.
\end{proof}

We have the following corollary.

\begin{corollary}\label{cortrunc}
Consider product weights  and $k\le s$. For $p>1$, we have
\[
  e(Q^{\rm trnc}_{s,n,k};F_{s,p,\bsgamma})\,\le\,
  \left([e(A_{k,n};F_{k,p,\bsgamma})]^{p^*}+\|\calI_s\|^{p^*}\,
\left(1-\exp
\left(\frac{-3}{2\,(p^*+1)}\sum_{j=k+1}^s \gamma_j^{p^*}\right)\right)
\right)^{1/p^*}
   \]
and for $p=1$, we have
\[
  e(Q^{\rm trnc}_{s,n,k};F_{s,1,\bsgamma})\,\le\,
  \max\left(e(A_{k,n};F_{k,1,\bsgamma})\,,\,\gamma_{k+1}
\right),
\]
where $\gamma_{k+1}=0$ if $k=s$.

Therefore, for the worst case error of $Q_{s,n,k}^{\rm trnc}$ not to
exceed the error demand $\e>0$, it is enough to choose $k=k(\e)$
so that
\begin{equation}\label{trunccrit}
1-\exp\left(\frac{-3}{2(p^*+1)}\,
  \sum_{j=k+1}^s \gamma_j^{p^*}\right)\,\le\,
  \frac{(\e/\|\calI_s\|)^{p^*}}{2},
\end{equation}
(or $\gamma_{k+1}\le\e$
for $p=1$),
and next to choose $n=n(\e)$ so that
\[
 e(A_{k,n};F_{k,1,\bsgamma})\,\le\,\frac{\e}{2^{1/p^*}}.
\]
\end{corollary}

Clearly the inequality \eqref{trunccrit} for $p>1$ is equivalent to
\begin{equation}\label{star}
  \sum_{j=k+1}^s \gamma_j^{p^*}\,\le\,-\frac{2\,(p^*+1)}3\,\log\left(1-
 \frac{(\e/\|\calI_s\|)^{p^*}}{2}\right).
\end{equation}

\begin{example}\label{ex:1}
Consider large $s$ including $s=\infty$ and
\[
\gamma_\setu\,=\,\prod_{j\in\setu}j^{-a}\quad\mbox{for\ }
a\,>\,1/p^*.
\]
Recall that then
\[
   \|\calI_s\|\,=\,\prod_{j=1}^s\left(1+\frac{j^{-ap^*}}{p^*+1}\right)^{1/p^*}
\]
for $p>1$ and $\|\calI_s\|=1$ for $p=1$.
Hence it is enough to take
\[
k(\e)\,=\,\left\lceil \e^{-1/a}-1\right\rceil\quad\mbox{for}\quad p\,=\,1.
\]
For $p>1$, we have
\[
  \frac{(k+1)^{-ap^*+1}}{ap^*-1}\,=\,\int_{k+1}^\infty x^{-ap^*}\rd x\,<\,
\sum_{j=k+1}^\infty j^{-ap^*}\,\le\,\int_{k+1/2}^\infty x^{-ap^*}\rd x\,=\,
\frac{(k+1/2)^{-ap^*+1}}{a\,p^*-1}.
\]
Therefore, to satisfy \eqref{star}, it is enough to take
\[
k\,=\,k(\e)\,=\,
\left\lceil\left(\frac{-3}{2 (p^*+1)\,(ap^*-1)\,\log(1-(\e/\|\calI_s\|)^{p^*}/2)}
\right)^{1/(ap^*-1)}-\frac12\right\rceil\quad\mbox{for}\quad
p>1.
\]

For $p=p^*=2$, which corresponds to the classical Hilbert space
setting, we have
\[
k(\e)\,=\,\left\lceil (-2\,(2\,a-1)\,\log(1-(\e/\|\calI_s\|)^2/2))^{-1/(2a-1)}
-\tfrac12\right\rceil.
\]
In calculating the values of $k(\e)$, we slightly overestimated the
norm of $\calI_s$ in the following way
\begin{eqnarray*}
  \|\calI_s\|^2 &\le& \prod_{j=1}^\infty \left(1+\frac{j^{-2a}}3\right)
   \,\le\,\prod_{j=1}^3\left(1+\frac{j^{-2a}}3\right)\,
   \exp\left(\sum_{j=4}^\infty\frac{j^{-2a}}3\right)\\
  &\le&\prod_{j=1}^3\left(1+\frac{j^{-2a}}3\right)\,
   \exp\left(\frac13\,\int_{3.5}^\infty x^{-2a}\rd x\right)\\
  &=&   \prod_{j=1}^3\left(1+\frac{j^{-2a}}3\right)\,
  \exp\left(\frac1{3\,(2a-1)}\,3.5^{-2a+1}\right).
\end{eqnarray*}
This gave us the following estimations for $\|\calI_s\|^2$
for $p=2$:
\[
   1.3703\ \mbox{for\ }a=2,\qquad 1.3411\ \mbox{for\ }a=3, \qquad
1.3352\ \mbox{for\ }a=4.
\]
Below are values of $k(\e)$ for $a=2,3,4$.
We have
\[
\begin{array}{c||c|c|c|c|c|c}
\e &  10^{-1}  &  10^{-2}  & 10^{-3}  & 10^{-4}  & 10^{-5}  &  10^{-6}\\
\hline
k(\e)& 4      &  17       & 77      & 357      & 1659    & 7701
\end{array}  \quad\mbox{for\ }a\,=\,2,
\]
\[
\begin{array}{c||c|c|c|c|c|c}
\e &  10^{-1}  &  10^{-2}  & 10^{-3}  & 10^{-4}  & 10^{-5}  &  10^{-6}\\
\hline
k(\e)& 2      &  6       &  12      & 31      & 77      & 193
\end{array}\quad\mbox{for\ }a=3,
\]
and
\[
\begin{array}{c||c|c|c|c|c|c}
\e &  10^{-1}  &  10^{-2}  & 10^{-3}  & 10^{-4}  & 10^{-5}  &  10^{-6}\\
\hline
k(\e)& 2      &  3        &  6     & 11      & 21       &  41
\end{array}\quad\mbox{for\ }a=4.
\]
It is clear that $k(\e)$ decreases with increasing $a$.
It also decreases when $p$ decreases as illustrated below for $p=1$:
\[
\begin{array}{c||c|c|c|c|c|c}
\e &  10^{-1}  &  10^{-2}  & 10^{-3}  & 10^{-4}  & 10^{-5}  &  10^{-6}\\
\hline
k(\e)& 3      &  9      &  31      & 99     & 316   & 999
\end{array}\quad\mbox{for\ }a=2,
\]
\[
\begin{array}{c||c|c|c|c|c|c}
\e &  10^{-1}  &  10^{-2}  & 10^{-3}  & 10^{-4}  & 10^{-5}  &  10^{-6}\\
\hline
k(\e)& 2      &  4        &  9     & 21      & 46   & 99
\end{array}\quad\mbox{for\ }a=3,
\]
and
\[
\begin{array}{c||c|c|c|c|c|c}
\e &  10^{-1}  &  10^{-2}  & 10^{-3}  & 10^{-4}  & 10^{-5}  &  10^{-6}\\
\hline
k(\e)& 1      &  3        &  5     &  9      & 17   &  31
\end{array}\quad\mbox{for\ }a=4.
\]
\end{example}

We end this section with the following remark concerning the
{\em normalized worst case error}. In a number of papers on
tractability of integration, instead of the {\em standard} worst
case error the normalized one is used. It is defined by
\[
   e^{\rm nrm}(Q_{s,n};F_{s,p,\bsgamma})\,:=\,
   \frac{e(Q_{s,n};F_{s,p,\bsgamma})}{\|\calI_s\|}.
\]
It follows clearly from Corollary \ref{cortrunc} that the following is
true for product weights.

\begin{corollary}\label{cor-nrm}
Consider product weights and $k\le s$. For $p>1$, we have
\[
e^{\rm nrm}(Q^{\rm trnc}_{s,n,k};F_{k,p,\bsgamma})\,\le\,
\left(\left[e^{\rm  nrm}(A_{k,n};F_{k,p,\bsgamma})\,
  \frac{\|\calI_k\|}{\|\calI_s\|}\right]^{p^*}+1-
  \exp\left(\frac{-3}{2(p^*+1)}\,\sum_{j=k+1}^s\gamma_j^{p^*}\right)
  \right)^{1/p^*}.
\]
Note that for any $p$ and any weights
\[
   \frac{\|\calI_k\|}{\|\calI_s\|}\,\le\,1.
\]
\end{corollary}

The corresponding numbers $k^{\rm nrm}(\e)$ for which the normalized
truncation error is bounded by $\e/2^{1/p^*}$,
\[
\left(1-\exp\left(\frac{-3}{2(p^*+1)}\,\sum_{j=k^{\rm nrm}(\e)+1}^s
\gamma_j^{p^*}\right)\right)^{1/p^*}
  \,\le\,\frac{\e}{2^{1/p^*}},
\]
are smaller than $k(\e)$
since for product weights $\|\calI_s\|> 1$. For instance, for $p=2$
and $\gamma_j^{-a}$, we have
\[
\begin{array}{c||c|c|c|c|c|c}
\e &  10^{-1}  &  10^{-2}  & 10^{-3}  & 10^{-4}  & 10^{-5}  &  10^{-6}\\
\hline
k^{\rm nrm}(\e)& 3      &  15      &  69      & 322     & 1494   & 6334
\end{array}  \quad\mbox{for\ }a\,=\,2,
\]
\[
\begin{array}{c||c|c|c|c|c|c}
\e &  10^{-1}  &  10^{-2}  & 10^{-3}  & 10^{-4}  & 10^{-5}  &  10^{-6}\\
\hline
k^{\rm nrm}(\e)& 2      &  5        &  11     & 29      & 72   & 182
\end{array}\quad\mbox{for\ }a=3,
\]
and
\[
\begin{array}{c||c|c|c|c|c|c}
\e &  10^{-1}  &  10^{-2}  & 10^{-3}  & 10^{-4}  & 10^{-5}  &  10^{-6}\\
\hline
k^{\rm nrm}(\e)& 1      &  3        &  5     & 11      & 20   & 39
\end{array}\quad\mbox{for\ }a=4.
\]

\section{ANOVA Decomposition and Truncation}
It is well known, see, e.g., \cite{KSWW10a}, that functions $h\in
H_{s,p,\bsgamma}$ also have a unique decomposition
\[
  h\,=\,\sum_{\setu\in\setU} h_\setu,
\]
where each $h_\setu$ depends only on the variables $x_j$ for
$j\in\setu$, and
\[
  \int_0^1 h_\setu(\bsx)\rd x_j\,=\,0\quad \mbox{if\ }j\in\setu.
\]
Unfortunately, unlike in the anchored decomposition, the terms
$h_\setu$ and $\sum_{\setu\subseteq[k]}h_\setu$ ($k<s$) cannot
be sampled.
This means that the truncation approach presented in
the previous section would not work in general since one cannot
get sharp estimations of the worst case truncation error
\[
  \sup_{\|h\|_{H_{s,p,\bsgamma}}\le 1}\left|\calI_s(h)-\int_{[0,1]^k}
   h(x_1,\dots,x_k,0,\dots,0)\rd(x_1,\dots,x_k)\right|.
\]

However it works
for product weights with sufficiently fast decaying $\gamma_j$'s. This
is why we assume for the rest of the paper that the weights have the
product form.

For product weights, the spaces $F_{s,p,\bsgamma}$ and
$H_{s,p,\bsgamma}$ (as sets of functions) are equal, see \cite{HeRiWa15}.
Moreover the embedding
\[
  \imath_{s,p,\bsgamma}: F_{s,p,\bsgamma}\hookrightarrow H_{s,p,\bsgamma}
\]
and its inverse
\[
  \imath_{s,p,\bsgamma}^{-1}: H_{s,p,\bsgamma}\hookrightarrow F_{s,p,\bsgamma}
\]
are bounded\footnote{Let $F$ and $H$ be normed spaces with norm
$\|\cdot\|_F$ and $\|\cdot\|_H$, respectively. We say that $F$ is
continuously embedded in $H$ and write $F \hookrightarrow H$, if
$F \subseteq H$ and if the inclusion map $\imath:F\rightarrow H$,
$x\mapsto x$, is continuous, i.e., if there exists some $C>0$ such
that $\|x\|_H \le C \|x\|_F$ for all $x \in F$.}. Indeed, it was
shown in \cite{HeRi13} that for $p=2$ we have
\[
   \max\left(\|\imath_{s,2,\bsgamma}\|\,,\,\|\imath_{s,2,\bsgamma}^{-1}\|
    \right)\,\le\,\prod_{j=1}^s\left(1+\frac{\gamma_j}{\sqrt{3}}
     +\frac{\gamma_j^2}3\right)^{1/2}.
\]
Next, it was shown in \cite{HeRiWa15} that for $p=1$ and
$p=\infty$
\[
  \|\imath_{s,1,\bsgamma}\|\,=\,\|\imath_{s,1,\bsgamma}^{-1}\|\,=\,
   \prod_{j=1}^s\left(1+\gamma_j\right)\quad\mbox{and}\quad
 \|\imath_{s,\infty,\bsgamma}\|\,=\,\|\imath_{s,\infty,\bsgamma}^{-1}\|\,=\,
   \prod_{j=1}^s\left(1+\gamma_j/2\right).
\]
Finally, the authors of \cite{SH} showed, applying the theory of
interpolation to the above result, that for every
$p\in[1,\infty]$, we have
\[
  \max\left(\|\imath_{s,p,\bsgamma}\|\,,\,\|\imath_{s,p,\bsgamma}^{-1}\|
      \right)\,\le\,\prod_{j=1}^s\left(1+\gamma_j\right).
\]

The following theorem provides for $p=2$ the exact value of the norms of the
embeddings
$\imath_{s,p,\bsgamma}$ and $\imath^{-1}_{s,p,\bsgamma}$ and shows
that these norms are equal.

\begin{theorem}\label{thm:p2}
Consider product weights. For $p=2$,
\begin{equation}\label{equiv}
  \|\imath_{s,2,\bsgamma}\|\,=\,\|\imath_{s,2,\bsgamma}^{-1}\|\,
  =\,\prod_{j=1}^s\left(1+\frac{\gamma_j}{\sqrt{3}}\,
    \left(\sqrt{1+\frac{\gamma_j^2}{12}}
     +\frac{\gamma_j}{\sqrt{12}}\right)\right)^{1/2}.
\end{equation}
Moreover
\begin{equation}\label{eq:new}
  1+\frac{\gamma_j}{\sqrt{3}}+\frac{\gamma_j^2}6\,\le\,
  1+\frac{\gamma_j}{\sqrt{3}}\,
    \left(\sqrt{1+\frac{\gamma_j^2}{12}}
     +\frac{\gamma_j}{\sqrt{12}}\right)\,\le\,
  1+\frac{\gamma_j}{\sqrt{3}}+\frac{\gamma_j^2}6+
  \frac{\gamma_j^3}{24\,\sqrt{3}}.
\end{equation}
\end{theorem}

\begin{proof}
Since the spaces $H_{s,2,\bsgamma}$ and $F_{s,2,\bsgamma}$ are
tensor products of the corresponding spaces of univariate
functions, it is enough to prove \eqref{equiv} for $s=1$ and a
generic weight $g\in(0,1]$. Moreover we will only consider
$\|f\|_{H_{1,2,g}}/\|f\|_{F_{1,2,g}}$ since the proof for
$\|f\|_{F_{1,2,g}}/\|f\|_{H_{1,2,g}}$ is very similar.

Note that for $f\equiv c$,
$\|f\|_{H_{1,2,g}}/\|f\|_{F_{1,2,g}}=1$. Hence it is enough to
consider
\[
  f(x)\,=\,\frac{c}{g}+\int_0^1h(t)\,1_{[0,x)}(t) \rd t
\]
for some $c\ge 0$ and $\|h\|_{L_2}=1$. Then
\[
   \|f\|_{F_{1,2,g}}^2\,=\,g^{-2}\,\left(1+c^2\right).
\]
Moreover
\begin{equation}\label{help}
   \int_0^1f(x)\rd x\,=\,\frac{c}g+\int_0^1h(t)\,(1-t)\rd t
   \,\le\,
   \frac1g\,\left(c+\frac{g\,\|h\|_{L_2}}{\sqrt{3}}\right)
   \,=\,
   \frac1g\,\left(c+\frac{g}{\sqrt{3}}\right)
\end{equation}
and, therefore,
\[
   \|f\|_{H_{1,2,g}}^2\,\le\,g^{-2}\,
  \left(\left(c+\frac{g}{\sqrt{3}}\right)^2+1\right).
\]
Hence
\[
  \frac{\|f\|_{H_{1,2,g}}^2}{\|f\|_{F_{1,2,g}}^2}\,\le\,
  \frac{(c+g/\sqrt{3})^2+1}{c^2+1}\,=\,
   \frac{c^2+1+2\,g\,c/\sqrt{3}+g^2/3}{c^2+1}\,=\,
  1+\frac{g}{\sqrt{3}}\,\rho(c;g),
\]
where
\[
  \rho(c;g)\,=\,\frac{2\,c+g/\sqrt{3}}{c^2+1}.
\]
It is easy to verify that
\[
   \max_{c\ge 0}\rho(c;g)\,=\,\rho(c^*_g;g),
  \quad\mbox{where}\quad c^*_g\,=\,\sqrt{1+\frac{g^2}{12}}
    -\frac{g}{\sqrt{12}}
\]
and then
\[
  \frac{\|f\|_{H_{1,2,g}}^2}{\|f\|_{F_{1,2,g}}^2}\,\le\,1+
   \frac{g}{\sqrt{3}}\,\rho(c^*_g;g).
\]
This shows that $\|\imath_{1,2,g}\|\le1+g/\sqrt{3}\,
\rho(c^*_g;g)$. To prove equality it is enough to notice that for
$h(t)=\sqrt{3}\,(1-t)$ we have equality in \eqref{help}, i.e.,
\[
   \int_0^1 h(t)\,(1-t)\rd t\,=\,\frac1{\sqrt{3}}.
\]
This proves that
\begin{equation}\label{claim1}
  \|\imath_{1,2,g}\|^2\,=\,\left(1+\frac{g}{\sqrt{3}}\,
    \rho(c^*_g;g)\right).
\end{equation}

We now show that $\rho(c^*_g;g)=\sqrt{1+g^2/12}+g/\sqrt{12}$.
It is easy to verify that
\[
  \rho(c_g^*;g)\,=\,\frac{\sqrt{1+g^2/12}}{1+g^2/12-\sqrt{1+g^2/12}\,
  g/\sqrt{12}}\,=\,\frac1{\sqrt{1+g^2/12}-g/\sqrt{12}}.
\]
Therefore, applying the conjugate to the last fraction we get
\begin{equation}\label{claim2}
  \rho(c^*_g;g)=\frac1{\sqrt{1+g^2/12}-\sqrt{g^2/12}}\cdot
   \frac{\sqrt{1+g^2/12}+g/\sqrt{12}}{\sqrt{1+g^2/12}+g/\sqrt{12}}
   \,=\,\sqrt{1+\frac{g^2}{12}}+\frac{g}{\sqrt{12}}.
\end{equation}

We now prove \eqref{eq:new}. The first inequality is trivial.
Clearly
\begin{eqnarray*}
 \frac{g}{\sqrt{3}}\,\left(\sqrt{1+\frac{g^2}{12}}+
   \frac{g}{\sqrt{12}}\right)
  &=&\frac{g}{\sqrt{3}}+\frac{g}{\sqrt{3}}\,\left(
   \frac{g}{\sqrt{12}}+ \sqrt{1+\frac{g^2}{12}}-1\right)\\
  &=& \frac{g}{\sqrt{3}}+\frac{g^2}{6}+E(g),
\end{eqnarray*}
where
\[
  E(g)\,:=\,\frac{g}{\sqrt{3}}\,\left(\sqrt{1+\frac{g^2}{12}}-1\right).
\]
Of course the term $E(g)$ is nonnegative, a can be
upper-bounded by
\[
  E(g)\,=\,\frac{g}{\sqrt{3}}\,\frac{g^2/12}{\sqrt{1+g^2/12}+1}\,=\,
   \frac{g^3}{24\,\sqrt{3}}\,\frac2{\sqrt{1+g^2/12}+1}\,\le\,
  \frac{g^3}{24\,\sqrt{3}}.
\]
Consequently,
\[
  1+\frac{g}{\sqrt{3}}\,\left(\sqrt{1+\frac{g^2}{12}}+\frac{g}{\sqrt{12}}
   \right)\,\le\,1+\frac{g}{\sqrt{3}}+\frac{g^2}{6}+
  \frac{g^3}{24\,\sqrt{3}},
\]
which completes the proof.
\end{proof}

The importance of 
the fact that the corresponding embeddings are bounded is
captured by the following corollary.

\begin{corollary}\label{coremb}
For every integration rule $Q_{s,n}$ we have
\[
   e(Q_{s,n};F_{s,p,\bsgamma}) \,\le\, \|\imath_{s,p,\bsgamma}\|
   \,e(Q_{s,n}; H_{s,p,\bsgamma})\quad\mbox{and}\quad
  e(Q_{s,n}; H_{s,p,\bsgamma})\,\le\,\|\imath_{s,p,\bsgamma}^{-1}\|\,
  e(Q_{s,n};F_{s,p,\bsgamma}).
\]
\end{corollary}

The essence of this corollary is that, for small
$\max(\|\imath_{s,p,\bsgamma}\|,\|\imath_{s,p,\bsgamma}^{-1}\|)$,
 an algorithm with small
worst case error with respect to one space has also small
worst case error with respect to the other space. In particular,
a good truncation in the space $F_{s,p,\bsgamma}$ leads to
efficient algorithms for $H_{s,p,\bsgamma}$.

We end this section with the following remark.

\begin{remark}
If
\[
  \sum_{j=1}^\infty \gamma_j\,<\,\infty,
\]
then the norms of the embedding operators are bounded
independently of $s$.
\end{remark}

\section{CBC Construction of Folded Lattice Rules for Integration in $F_{s,2,\bsgamma}$ and $H_{s,2,\bsgamma}$}\label{seccbc}

Now we consider folded (also called tent transformed) lattice rules.
Throughout this section we only consider product weights
and $p=2$. Note that then $p^*=2$ and for the exponent in
Proposition~\ref{proptrunc}
\[
   \frac{3}{2\,(p^*+1)}\,=\,\frac12.
\]

For $n \in \mathbb{N}$ and $\bsz \in \mathbb{Z}^s$ a lattice
rule with $n$ points and generating vector $\bsz$ is a quadrature
rule of the form
\begin{equation}\label{def_LR}
A_{n,s}(\bsz)(f)=\frac{1}{n}\sum_{k=0}^{n-1}f\left(\left\{
\frac{k}{n} \bsz \right\}  \right),
\end{equation}
where the fractional part $\{(k/n)\bsz\}$ of $ $ is meant
component-wise. Lattice rules are especially suited for the integration of
1-periodic, smooth functions (e.g. from Korobov spaces), for which
there exist excellent error estimates~\cite{DKS,LP14,niesiam,slojoe}.
These results can also be
transferred to non-periodic functions when one replaces lattice
rules by  folded (or tent transformed) lattice rules. The tent
transform $\phi:[0,1] \rightarrow [0,1]$ is a Lebesgue measure
preserving map given by $\phi(x) = 1 -
|1-2x|$. For a vector $\bsx \in [0,1]^s$ let $\phi(\bsx)$ be
defined component-wise. Then the folded version of \eqref{def_LR}
is given by
\begin{equation}\label{def_LR_fold}
A_{n,s}^{\phi}(\bsz)(f)=\frac{1}{n}\sum_{k=0}^{n-1}f\left(\phi\left(\left\{
\frac{k}{n} \bsz \right\} \right) \right).
\end{equation}
The idea of using folded lattice rules was first introduced by
Hickernell in \cite{H02}.

For the worst case error of a folded lattice rule in the
unanchored Sobolev space $H_{s,2,\bsgamma}$ it follows from
\cite[Lemma~1 and lines 11--13 on page 277]{DNP14} that
\begin{equation}\label{eq_wce_tent}
e(A_{n,s}^{\phi}(\bsz);H_{s,2,\bsgamma}) \le e(A_{n,s}(\bsz);H^{\rm
Kor}_{s,2,\pi^{-2} \bsgamma}),
\end{equation}
where $\pi^{-2} \bsgamma = (\pi^{-2|\setu|} \gamma_{\setu})_{\setu
\subseteq \sset}$. For $\alpha > 1$ the Korobov space $H^{\rm
Kor}_{s,\alpha,\bsgamma}$ is a reproducing kernel Hilbert space of
1-periodic functions with kernel function
\[
  K_{s,\alpha,\bsgamma}(\bsx,\bsy)=\sum_{\bsh \in \mathbb{Z}^s}
r(\bsh) \exp(2 \pi \mathtt{i} \bsh \cdot (\bsx-\bsy)).
\]
 Here, for $\bsh=(h_1,h_2,\ldots,h_s) \in \mathbb{Z}^s$,
$r(\bsh)=\prod_{j=1}^s r_j(h_j)$, and for $h\in \mathbb{Z}$ we put
\[
r_j(h)=\left\{
\begin{array}{ll}
1 & \mbox{ if } h=0,\\
\frac{\gamma_j^2}{|h|^{\alpha}} & \mbox{ if } h \not=0.
\end{array}\right.
\]

Hence the worst case error of a lattice rule in
$H^{{\rm Kor}}_{s,2,\bsgamma}$
dominates the worst case error of the folded version of the same
lattice rule in $H_{s,2,\pi^2 \bsgamma}$ (whose elements are not
necessarily 1-periodic).

There are a lot of results concerning the worst case error of
lattice rules for Korobov spaces. Excellent generating vectors can
be constructed component-wise with so-called component by
component (or, for short, CBC) algorithms.
The CBC approach goes back to Korobov~\cite{kor} in the 1960s.
Later it was re-invented
by Sloan and Reztsov~\cite{SR} in 2002 and became a powerful tool
in constructing lattice rules for high-dimensional
problems. We refer to \cite{dick04,kuo,SR,SKJ} for the
CBC construction and \cite{NC06,CoolNuy} for the fast CBC construction
according to Cools and Nuyens.

For example, for product weights we have the following result
which is essentially \cite[Theorem~5.12]{DKS}.

\begin{theorem}\label{thmCBClatkor}
Let $n$ be a prime number and consider product weights
$\boldsymbol{\xi}=\prod_{j\in\setu}\xi_j$. One can
construct with a fast
CBC algorithm a lattice point $\bsz \in \{0,1,\ldots,n-1\}^s$ such
that
\begin{equation}\label{cbc_wc_err}
e(A_{n,s}(\bsz);H^{\rm Kor}_{s,\alpha,\boldsymbol{\xi}})\le
\frac{1}{(n-1)^{1/(2\lambda)}}\left(-1+\prod_{j=1}^s (1+2
\xi_j^{2\lambda} \zeta(\lambda \alpha))\right)^{1/(2\lambda)}
\end{equation}
for all $\lambda \in (1/\alpha,1]$, where $\zeta$ denotes the
Riemann zeta function, $\zeta(\alpha)=\sum_{j\ge 1} j^{-\alpha}$.
The construction cost of the
fast CBC algorithm is of order of magnitude $O(s n \log n)$.
\end{theorem}
>From Theorem~\ref{thmCBClatkor} in conjunction with
\eqref{eq_wce_tent} and Corollary~\ref{coremb} we obtain the following result.
\begin{theorem}\label{thmcbcresult}
Let $s\in \mathbb{N}$ be given and let $n$ be a prime number. If $\bsz \in
\{0,1,\ldots,n-1\}^s$ is constructed such that \eqref{cbc_wc_err}
holds, then for all $\lambda \in (1/2,1]$ we have
\[
  e(A_{n,s}^{\phi}(\bsz);H_{s,2,\bsgamma}) \le \frac{1}{(n-1)^{1/(2\lambda)}}
\left(-1+\prod_{j=1}^s \left(1+2 \left(\frac{\gamma_j}{\pi^2}
\right)^{2 \lambda} \zeta(2\lambda)\right)\right)^{1/(2\lambda)}
\]
 and
\[
  e(A_{n,s}^{\phi}(\bsz);F_{s,2,\bsgamma}) \le \|\imath_{s,2,\bsgamma}\|
\frac{1}{(n-1)^{1/(2\lambda)}}\left(-1+\prod_{j=1}^s (1+2
\left(\frac{\gamma_j}{\pi^2}\right)^{2 \lambda}
\zeta(2\lambda))\right)^{1/(2\lambda)}.
\]
\end{theorem}

Hence we have a fast CBC construction of quadrature rules for the
unanchored space $H_{s,2,\bsgamma}$ and for the anchored space
$F_{s,2,\bsgamma}$.

\section{Truncated Quadrature Rule for
$H_{s,2,\bsgamma}$ and  $F_{s,2,\bsgamma}$
Based on Folded Lattice Rules}\label{sectrunc}

Now we combine Theorem \ref{thmcbcresult} with the truncation in
the sense of $F_{s,2,\bsgamma}$. For the $k$-dimensional
quadrature rules we use folded lattice rules satisfying
\eqref{cbc_wc_err} with $s$ replaced by $k$. We deduce the
following theorem from Corollary \ref{cortrunc}, Corollary
\ref{coremb}, and Theorem \ref{thmcbcresult}.

\begin{theorem}
Let $s\in\bbN$ be given and let $n$ be a prime number. Let $k\in\bbN$
be chosen such that $k\le s$. If $\bsz\in\{0,1,\ldots,n-1\}^k$ is
constructed such that
\eqref{cbc_wc_err} with $s$ replaced by $k$ holds, then for all
$\lambda\in (1/2,1]$ we have
\begin{eqnarray}\label{thmtruncF}
e(Q^{\rm trnc}_{s,n,k};F_{s,2,\bsgamma})&\le&
  \left[\frac{\|\imath_{k,2,\bsgamma}\|^2}{(n-1)^{1/\lambda}}
\left(-1+\prod_{j=1}^k \left(1+2 \left(\frac{\gamma_j}{\pi^2}
  \right)^{2\, \lambda}
\zeta(2\lambda)\right)\right)^{1/\lambda}\right.\nonumber\\
&&\left. + \|\calI_s\|^2\, \left(1-\exp\left(-\frac{1}2\,\sum_{j=k+1}^s
\gamma_j^{2}\right)\right)
\right]^{1/2}.
\end{eqnarray}
and
\begin{eqnarray}\label{thmtruncH}
e(Q^{\rm trnc}_{s,n,k};H_{s,2,\bsgamma}) & \le &
\|\imath_{s,2,\bsgamma}^{-1}\|
\left[\frac{\|\imath_{k,2,\bsgamma}\|^2}{(n-1)^{1/\lambda}}
\left(-1+\prod_{j=1}^k \left(1+2 \left(\frac{\gamma_j}{\pi^2}\right)^{2
      \lambda}
\zeta(2\lambda )\right)\right)^{1/\lambda}\right.\nonumber\\
&&\hspace{1.3cm}+\|\calI_s\|^2\,
\left.\left(1-\exp\left(-\frac{1}2\,\sum_{j=k+1}^s
\gamma_j^{2}\right)\right) \right]^{1/2}.
\end{eqnarray}
\end{theorem}

\begin{remark}\rm
 Note that the truncated quadrature rule $Q^{\rm trnc}_{s,n,k}$ in
 \eqref{thmtruncF} and \eqref{thmtruncH}, respectively,
can be constructed using $O(kn\log n)$ (as opposed to $O(sn\log n)$)
operations.
\end{remark}

Let us now discuss the bounds in \eqref{thmtruncF} and
\eqref{thmtruncH}. We assume that $s$ is huge and that the product weights
$\bsgamma$ satisfy the condition
\begin{equation}\label{eqgnonsqared}
 \sum_{j=1}^\infty \gamma_j <\infty.
\end{equation}
Furthermore, recall that $\gamma_j \le 1$ for all $j \in \mathbb{N}$.
From this 
and, by Theorem \ref{thm:p2},  and standard arguments, we get
\begin{eqnarray*}
\|\imath_{s,2,\bsgamma}\|^2 = \|\imath_{s,2,\bsgamma}^{-1}\|^2
& \le &\prod_{j=1}^s \left(1+\gamma_j \left(\frac1{\sqrt{3}}
   +\frac1{6}+\frac{1}{24 \sqrt{3}}\right)\right)\\
&\le & \prod_{j=1}^s \left(1+0.7681\cdot\gamma_j\right)\\
&\le &\exp \left(0.7681\cdot\sum_{j=1}^\infty\gamma_j\right)
\,=:\,C_1 (\bsgamma).
\end{eqnarray*}

Similarly, we see that
\[
\prod_{j=1}^k \left(1+2 \left(\frac{\gamma_j}{\pi^2}\right)^{2 \lambda}
\zeta(2\lambda)\right)\le
\exp\left(\frac{2}{\pi^{4\,\lambda}}\,\zeta (2\lambda) \sum_{j=1}^\infty
  \gamma_j^{2\,\lambda}\right)\,=:\, C_2 (\bsgamma,\lambda),
\]
and
\[
  \|\mathcal{I}_s\|^2\, =\, \prod_{j=1}^s \left(1+
\frac{\gamma_j^2}{3}\right) \le \exp\left(\frac{1}{3}\,
\sum_{j=1}^\infty\gamma_j^2\right) \,\le\, C_1 (\bsgamma).
\]
In summary, we obtain from \eqref{thmtruncF} that
\begin{equation}\label{intermediateF}
[e(Q^{\rm trnc}_{s,n,k};F_{s,2,\bsgamma})]^2\, \le \,
\frac{C_1(\bsgamma) C_2^{1/\lambda}
  (\bsgamma,\lambda)}{(n-1)^{1/\lambda}}
+ C_1(\bsgamma)
\left(1-\exp\left(-\frac{1}2\,\sum_{j=k+1}^s\gamma_j^{2}\right)\right),
\end{equation}
and from \eqref{thmtruncH} that
\begin{eqnarray*}
[e(Q^{\rm trnc}_{s,n,k};H_{s,2,\bsgamma})]^2 &\le&
\frac{C_1^2 (\bsgamma)\, C_2^{1/\lambda}(\bsgamma,\lambda)}{(n-1)^{1/\lambda}}
+C_1^2(\bsgamma)
\left(1-\exp\left(-\frac{1}2\,\sum_{j=k+1}^s\gamma_j^{2}\right)\right)\\
& \le & \frac{C_4(\bsgamma,\lambda)}{n^{1/\lambda}}+C_1^2(\bsgamma)
\left(1-\exp\left(-\frac{1}2\,\sum_{j=k+1}^s\gamma_j^{2}\right)\right),
\end{eqnarray*}
where
\[
C_4(\bsgamma,\lambda)\,=\,2^{1/\lambda}\,C_1^2(\bsgamma)\,
C_2^{1/\lambda} (\bsgamma,\lambda).
\]

Let now $G:\mathbb{R}^+ \rightarrow \mathbb{R}^+$ be a strictly
decreasing bijective
function with $\lim_{x \rightarrow \infty}G(x)=0$ such that
\[
  G(k) \ge \sum_{j=k+1}^{\infty}\gamma_j^2 \ \ \ \ \mbox{ for $k \in
  \mathbb{N}$.}
\]
Note that $G$ exists due to the assumption in \eqref{eqgnonsqared}
and also $G^{-1}$ exists and is strictly decreasing as well. Then we
obtain from \eqref{intermediateF} that
\[
   [e(Q^{\rm trnc}_{s,n,k};F_{s,2,\bsgamma})]^2 \le C_1(\bsgamma)\,
\left(\left(\frac{2\,C_2(\bsgamma,\lambda)}{n}\right)^{1/\lambda}
   +
\left(1-{\rm e}^{-G(k)/2}\right)\right).
\]
Now we choose $k$ such that
\[
\left(\frac{2\, C_2(\bsgamma,\lambda)}{n}\right)^{1/\lambda}
  \, \asymp\, 1-{\rm e}^{-G(k)/2}.
\]
For $n>2\,C_2(\bsgamma,\lambda)$, this is satisfied if
\[
G(k)\,\asymp\, - 2\,\log\left(1-
\left(\frac{2\,C_2(\bsgamma,\lambda)}n
    \right)^{1/\lambda}\right)
  \,\asymp\,2 \left(\frac{2\,C_2(\bsgamma,\lambda)}{n}\right)^{1/\lambda}.
\]
Hence
\[
k \asymp G^{-1}\left(2\,(2\,C_2(\bsgamma,\lambda)/n)^{1/\lambda}\right).
\]
This means that for $\lambda \in (1/2,1]$  we obtain an error of
order of magnitude
$$
e(Q^{\rm trnc}_{s,n,k};F_{s,2,\bsgamma}) \ll_{\bsgamma,\lambda}
\frac{1}{n^{1/(2 \lambda)}}
$$
under a construction cost of order of magnitude
\[
O\left(n \, G^{-1}((2\,(2\,C_2(\bsgamma,\lambda)/n))^{-1/\lambda})\,  \log n\right)
\]
for $s$ arbitrarily large. A similar assertion holds for
$e(Q^{\rm trnc}_{s,n,k};H_{s,2,\bsgamma})$.

We end this section with the following examples.
\begin{example}\rm
Assume that $\gamma_j=j^{-a}$ with $a >1$. Then we have
$$
  \sum_{j=k+1}^s\gamma_j^2\,\le\, \int_k^{\infty} \frac{1}{t^{2 a}} \rd t
\,=\,\frac{1}{2a -1} \frac{1}{k^{2a-1}}.
$$
Hence we choose $$G(x)=\frac{1}{2a-1} \frac{1}{x^{2 a-1}}$$ and
therefore
$$
  G^{-1}(x)=\left(\frac{1}{2 a-1} \frac{1}{x}\right)^{1/(2a-1)}.
$$
This means that for $\lambda \in (1/2,1]$  we obtain an error of
order of magnitude
$$
  e(Q^{\rm trnc}_{s,n,k};F_{s,2,\bsgamma}) \ll_{\bsgamma,\lambda}
\frac{1}{n^{1/(2 \lambda)}}
$$
under a construction cost of order of magnitude
$$
  O\left(n^{1+\frac{1}{\lambda (2 a-1)}} \log n\right)
$$
for $s$ arbitrarily large. The same assertion holds for
$e(Q^{\rm trnc}_{s,n,k};H_{s,2,\bsgamma})$.
\end{example}

\begin{example}\rm
Assume now that $\gamma_j=q^j$ for $q\in(0,1)$. Then we can take
\[
   G(k)\,=\frac{q^{k+1}}{1-q}
\]
and
\[
  G^{-1}(x)\,=\,\frac{\log(1/(x\,(1-q)))}{\log(1/q)}.
\]
This means that for $\lambda \in (1/2,1]$  we obtain an error of
order of magnitude
$$
  e(Q^{\rm trnc}_{s,n,k};F_{s,2,\bsgamma}) \ll_{\bsgamma,\lambda}
\frac{1}{n^{1/(2 \lambda)}}
$$
under a construction cost of order of magnitude
\[
  O\left(n\,\log^2(n)\right)
\]
for $s$ arbitrarily large. The same assertion holds for
$e(Q^{\rm trnc}_{s,n,k};H_{s,2,\bsgamma})$.
\end{example}

\section{Generalizations}
For simplicity of discussion, we presented so far the results for the domain
$D=[0,1]$, the standard $L_p$ norm
\[
  \|g\|_{L_p}\,=\,\left(\int_0^1|g(x)|^p\rd x\right)^{1/p},
\]
and the {\em un-weighted} integration problem of approximating
$\calI_s$.

However, the results of \cite{HeRiWa15,SH} on the equivalence of
anchored and ANOVA spaces hold for more general domains and
norms, as shown in \cite{GnHeHiRiWa15}. Our results Theorem
\ref{thm-main}, Proposition \ref{proptrunc}, Corollaries
\ref{cortrunc} and \ref{cor-nrm}, and Theorem \ref{thm:p2} can
easily be extended to this more general setting.

More specifically let $D$ be an interval
\[
  D\,=\,[0,T]\quad \mbox{or}\quad D\,=\,[0,\infty),
\]
and let
\[
  \psi:D\to\bbR_+
\]
be a positive (a.e.) probability density function. The authors of
\cite{GnHeHiRiWa15} provide a necessary and sufficient condition on
$p\in[1,\infty]$ and $\psi$ so that $F_{s,p,\bsgamma}$ and
$H_{s,p,\bsgamma}$ are well defined Banach spaces when endowed with the
norms
\[
  \|f\|_{F_{s,p,\bsgamma}}\,=\,\left(\sum_{\setu\in\setU}\gamma_\setu^{-p}\,
    \int_{D^{|\setu|}}
    \left|f^{(\setu)}([\bsx_\setu;\bszero_{-\setu}])\right|^p
   \,\psi_{\setu}(\bsx_\setu)\rd\bsx_\setu\right)^{1/p},
\]
and
\[
  \|f\|_{H_{s,p,\bsgamma}}\,=\,
  \left(\sum_{\setu\in\setU}\gamma_\setu^{-p}\,\int_{D^{|\setu|}}
   \left|\int_{D^{s-|\setu|}}f^{(\setu)}([\bsx_\setu;\bsx_{-\setu}])\,\psi_{-|\setu|}
   (\bsx_{-\setu})\rd\bsx_{-\setu}\right|^p\psi_\setu(\bsx_\setu)\rd
   \bsx_\setu\right)^{1/p}
\]
respectively. Here
\[
  \psi_\setu(\bsx_\setu)\,=\,\prod_{j\in\setu}\psi(x_j).
\]

For product weights they show that
\[
  \|\imath_{s,1,\bsgamma}\|\,=\,\|\imath_{s,1,\bsgamma}^{-1}\|\,=\,
   \prod_{j\in\setu}(1+\gamma_j\,\kappa_\psi)\quad\mbox{and}\quad
 \|\imath_{s,\infty,\bsgamma}\|\,=\,\|\imath_{s,\infty,\bsgamma}^{-1}\|\,=\,
   \prod_{j\in\setu}(1+\gamma_j\,m_\psi),
\]
where
\[
    m_\psi\,=\,\int_D x\,\psi(x)\rd x\quad\mbox{and}\quad
    \kappa_\psi\,=\,\esup_{x\in D} \frac{\int_D  (t-x)^0_+\psi(t)\rd t}{\psi(x)},
\]
and the upper bounds for $p \in (1,\infty)$ obtained via interpolation theory, see \cite{GnHeHiRiWa15} for more.

Let
\[
   \rho:D\to\bbR_+
\]
be a probability density function and let
\[
  \rho_s:D^s\to\bbR_+\quad\mbox{be defined by}\quad
\rho_s(\bsx)\,=\,\prod_{j=1}^s \rho(x_j).
\]
Consider now the integration problem of approximating
\[
  \calI_{s,\rho}(f)\,=\,\int_{D^s} f(\bsx)\,\rho_s(\bsx)\rd\bsx
\]
for $f\in F_{s,p,\bsgamma}$ or $f\in H_{s,p,\bsgamma}$.

It is easy to verify that
\[
  \left|\calI_s(f_\setu)\right|\,\le\,\|f^{(\setu)}([\cdot_\setu;
  \bszero_{-\setu}])\|_{L_{p.\psi}(D^{\setu})}\,\|I_1\|^{|\setu|},
\]
where
\[
  \|I_1\|\,=\,\left(\int_D\psi^{-p^*/p}(t)\,\left|
   \int_D (x-t)^0_+\,\rho(x)\rd x\right|^{p^*}\right)^{1/p^*}
\]
Of course, for $p=1$, we have
\[
  \|I_1\|\,=\,\esup_{t\in D}\int_D(x-t)^0_+\,\rho(x)\rd x\,=\,1.
\]
Therefore
\[
  \|\calI_s\|\,=\,\left(\sum_{\setu\in\setU} \gamma_\setu^{p^*}\,
  \|I_1\|^{p^*\,|\setu|}\right)^{1/p^*}.
\]
Assuming that $\calI_s$ is continuous, i.e., $\|\calI_s\|<\infty$, the results of Section 3 hold
with
\[
   \frac1{p^*+1}\quad\mbox{replaced by}\quad\|I_1\|^{p^*}.
\]

Finally we add that similar positive results for the dimension truncation
can be obtained in other than the {\em worst case} settings and for
other than integration problems including function approximation.
These generalizations will be presented in our future papers.


\begin{small}
\noindent\textbf{Authors' addresses:}\\

\medskip

\noindent Peter Kritzer, Friedrich Pillichshammer\\
Institut f\"{u}r Finanzmathematik und Angewandte Zahlentheorie,
Johannes Kepler Universit\"{a}t Linz\\
Altenbergerstr.~69, 4040 Linz, Austria\\
E-mail: \texttt{peter.kritzer@jku.at},
\texttt{friedrich.pillichshammer@jku.at}.

\medskip

\noindent G. W. Wasilkowski\\
Computer Science Department, University of Kentucky\\
301 David Marksbury Building\\
 Lexington, KY 40506, USA\\
E-mail: \texttt{greg@cs.uky.edu}
\end{small}
\end{document}